\documentclass[11pt]{article}
\usepackage[english]{babel}
\usepackage{amsmath} 	
\usepackage{amsfonts}
\usepackage{amsthm}
\usepackage{amssymb}
\usepackage{natbib}
\usepackage[utf8]{inputenc}
\usepackage{threeparttable}
\usepackage{mathtools}
\usepackage{footnote}
\usepackage[bottom]{footmisc}
\usepackage{graphicx}
\usepackage[font=small]{caption}
\usepackage[font=footnotesize]{subcaption}

\usepackage{comment}
\usepackage{enumerate}
\usepackage{chngcntr}
\counterwithout{paragraph}{subsubsection} % removes paragraph from the subsubsections
\usepackage{lettrine} % The lettrine is the first enlarged letter at the beginning of the text
\usepackage{paralist} % Used for the compactitem environment which makes bullet points with less space between them
\usepackage[hmarginratio=1:1,top=32mm,columnsep=20pt]{geometry} % Document margins

\theoremstyle{plain}
\newtheorem{theorem}{Theorem}[section]
\newtheorem{proposition}[theorem]{Proposition}
\newtheorem{lemma}[theorem]{Lemma}

\theoremstyle{definition}

\newtheorem{example}[theorem]{Example}
\theoremstyle{remark}

\def\pr{{\rm Pr}}

\newcommand{\eid}{\stackrel{d}{=}}

\usepackage{changes}
\definechangesauthor[name={Per cusse}, color=orange]{per}
\setremarkmarkup{(#2)}

\usepackage{lipsum}
\newcommand\blfootnote[1]{%
  \begingroup
  \renewcommand\thefootnote{}\footnote{#1}%
  \addtocounter{footnote}{-1}%
  \endgroup
}

\begin{document}

% ===========================
% title

\thispagestyle{empty}
\begin{center}
{\LARGE \textbf{Discrete Extremes}\\}
\phantom{adsf}
%{\large name\\}
{\large Adrien Hitz, Richard Davis and Gennady Samorodnitsky\\}
%{\large \textit{University of Oxford}\\}
\end{center}

\blfootnote{Adrien Hitz, University of Oxford, 24-29 St Giles, Oxford OX1 3LB, UK.} 
\blfootnote{Gennady Samorodnitsky, Cornell University, 220 Rhodes Hall, Ithaca, NY.} 
\blfootnote{Richard Davis, Columbia University, 1255 Amsterdam Avenue, New York, NY.}

% ===========================
% Abstract

\small
\setlength{\leftskip}{1cm}
\setlength{\rightskip}{1cm}

Our contribution is to widen the scope of extreme value analysis applied to discrete-valued data. Extreme values of a random variable $X$ are commonly modeled using the generalized Pareto distribution, a method that often gives good results in practice. When $X$ is discrete, we propose two other methods using a discrete generalized Pareto and a generalized Zipf distribution respectively. Both are theoretically motivated and we show that they perform well in estimating rare events in several simulated and real data cases such as word frequency, tornado outbreaks and multiple births. 

% ===========================
% Main 

\setlength{\leftskip}{0pt}
\setlength{\rightskip}{0cm}
\normalsize

% ===========================
% Warnings, todo list, questions and other ideas

% ===========================
% Introduction

\section{Introduction}\label{par:intro} 

Extreme quantile estimation is an important but difficult problem in statistics, especially when the quantile is beyond the range of the data. In the univariate case, extreme value theory motivates the choice of a parametric family called the generalized Pareto distribution (GPD) which is used to model the tail and estimate the probability of rare events (\cite{Pickands1975}). Let $X$ be a random variable taking values in $[0,x_F)$ with survival function $\bar F_X,$ where $x_F\in \mathbb R_+\cup \{\infty\}$ and $\mathbb R_+=(0,\infty).$ Suppose that there exists a strictly positive sequence $a_u$ such that
\begin{align}\label{eq:convG}
a_u^{-1}(X-u)\mid X\geq u \,\overset{d}{\rightarrow}\, Z,
\end{align}
as $u\rightarrow x_F,$ for some $Z$ following a non-degenerate probability distribution on $[0,\infty),$ where $\overset{d}{\rightarrow}$ denotes weak convergence. A stunning result is that this assumption is sufficient to characterize the limiting distribution: $Z$ follows a generalized Pareto distribution (GPD), defined by its survival function
 \begin{align*}
 1-F_{\text{GPD}}(x;\sigma,\xi)=
\bar F_{\text{GPD}}(x;\sigma,\xi)=\left(1+\xi {x\over \sigma}\right)^{-1/\xi} 1_{\{x<\tau\}},\quad x\geq 0,
 \end{align*}
 with $\tau=\infty$ if $\xi\geq 0,$ $\tau=\sigma/|\xi|$ if $\xi<0,$ where
 $1_{\{x<\tau\}}$ is $1$ if $x<\tau$ and $0$ otherwise, and $\sigma>0.$ We use the convention that if $\xi=0,$ then $(1+\xi x)^{1/\xi}=e^x$. Condition
 (\ref{eq:convG}), written as $X\in \text{MDA}_{\xi},$ means that
 $X$ is in the maximum domain of attraction of an extreme value distribution with shape parameter $\xi$
 (see e.g.\ \cite{Resnick1987}). Sometimes one  says that the law $F$ of $X$ is
 in $\text{MDA}_{\xi}$. In this case, the  sequence of cumulative distribution functions of $a_u^ {-1} (X-u)\mid X\geq u$ converges uniformly to $F_{\text{GPD}}$ on $[0,\infty).$ Thus, for large $u,$
 \begin{align}\label{eq:approxGPD}
 P(X-u> x\mid X\geq u) &=  P\{a_u^{-1} (X-u) >  a_u^{-1}
   x\mid X\geq u\} \notag \\
\shortintertext{is often approximated by}
&\approx  \bar F_{\text{GPD}}(x;\sigma a_u,\xi)\,, 
 \end{align}
motivating the approximation of the
distribution of exceedances of $X$ above a large threshold $u$ by a
GPD (\cite{Davison1990}). In the continuous case, where $X$ is assumed to have a continuous distribution, most common distributions belong to some maximum domain of attraction, 
 and the GPD approximation to the tail of the distribution can be applied. This approximation often works well in practice, although it can be poor when $u$ is not  large enough, for instance
 if $X$ is normal and $u$ is around the $90$th percentile of its distribution. 

Two issues are apparent with applying the GPD method of approximating the distribution
tail to a discrete distribution.  First, a necessary condition
for a discrete distribution $F_X$ to be in some maximum domain of attraction in the case $x_F=\infty$
is that $F_X$ is long-tailed,  i.e., ${\bar F_X(u+1)/ \bar F_X(u)}\rightarrow
1$ as $u\rightarrow \infty$, see \cite{Shimura2012} and
\cite{Anderson1970,Anderson1980}. Without being in a maximum domain of
attraction the GPD approximation \eqref{eq:approxGPD} does not necessarily apply. Note that many common discrete distributions,
including geometric, Poisson and negative binomial distributions, are not long-tailed. 
The second issue in approximating a discrete distribution by a GPD, a continuous distribution,  is that ties are not allowed. To overcome these limitations, we suggest   two
 alternative methods of modeling the tails of discrete observations,
 each one relying on a specific assumption on the underlying
 distribution. For this work, we will only consider approximations to distributions with infinite support.

We say that a discrete random variable $X$ with non-negative values is in the discrete maximum domain of attraction, which we write as $X\in
\text{D-MDA}_{\xi}$, if there exists a random
variable $Y\in \text{MDA}_\xi$ with $\xi\geq 0$ such that $\pr(X\geq k) =  \pr(Y\geq k)$ for $k=0,1,2,\ldots,$ i.e., the equality in
distribution, $X\eid \lfloor Y\rfloor$, holds. We call $Y$ an extension of $X$ and such an extension is not unique. \cite{Shimura2012} showed that $X\in \text{MDA}_\xi$ for
some $\xi\geq 0$ if and only if $X\in\text{D-MDA}_\xi$ and $X$ is
long-tailed. It was also shown by \cite{Shimura2012} that geometric, Poisson and
negative binomial distributions are in $\text{D-MDA}$. This
set is, therefore, strictly larger than the set of all discrete
distributions in $ \text{MDA}_\xi.$

Let $X\in \text{D-MDA}_\xi$ and $Y\in\text{MDA}_\xi$ be the corresponding extension satisfying $X\eid \lfloor Y\rfloor$. Then, for large integers $u,$
\begin{align}
P(X-u= k\mid X\geq u) &= P(Y-u\geq k\mid Y\geq u)-P(Y-u\geq k+1\mid
                          Y\geq u) \notag \\
\shortintertext{which we will approximate as in \eqref{eq:approxGPD} by}
& \approx   
% \bar F_{\text{GPD}}(k;\sigma a_u,\xi)-\bar F_{\text{GPD}}(k+1;\sigma a_u,\xi) \nonumber\\
 p_{\text{D-GPD}}(k;\sigma a_u,\xi), \label{eq:approxDGPD} 
\end{align}
 where $p_{\text{D-GPD}}$ is the probability mass function of the
discrete generalized Pareto distribution (D-GPD) defined by 
\begin{equation} \label{e:p.d.gpd}
p_{\text{D-GPD}}(k;\sigma,\xi)=\bar
F_{\text{GPD}}(k;\sigma,\xi)-\bar F_{\text{GPD}}(k+1;\sigma,\xi),
\end{equation}
for $k=0,1,2,\ldots$.
The discrete generalized Pareto distribution has been used by 
\cite{Prieto2014} to model road accidents, while various aspects of discrete Pareto-type distributions were studied in 
 \cite{Krishna2009}, \cite{Buddana2014}, and \cite{Kozubowski2015}. Notice that the scale parameter in (\ref{eq:approxDGPD}) is undetermined since the extension $Y$ is non-unique.

We now discuss an alternative assumption on the distribution of the discrete random
variable $X$ that will allow us to construct another approximation of its tail. Let $p_X$ be the probability mass function of $X$ and suppose that there exists a non-negative random variable $Y\in\text{MDA}_{\xi/(1+\xi)}$ with $\xi\geq 0$ such that $p_X(k)= c \,\bar F_Y(k)$ for $k=d, d+1,d+2,\ldots,$ for some $c>0$ and $d\in\mathbb N_0=\{0,1,\ldots\}.$ We denote this condition by $p_X\in \text{D-MDA}_{\xi/(1+\xi)}$ and call $\bar F_Y$ an extension of $p_X.$ Then for large integers $u$, 
\begin{align}
P(X-u= k\mid X\geq u) &= {p_X(u+k)/p_X(u) \over \sum_{i=0}^\infty
  p_X(u+i)/p_X(u) } \notag \\
&= \frac{P(Y>u+k)/P(Y>u)}{\sum_{i=0}^\infty P(Y>u+i)/P(Y>u)} \notag \\
\shortintertext{which can be approximated as in \eqref{eq:approxGPD},} 
 & \approx  p_{\text{GZD}}\bigl(k; (1+\xi)\sigma a_u,\xi\bigr)\,,\label{eq:approxZM}
\end{align}
where 
\begin{align}\label{eq:zmrepa}
p_{\text{GZD}}(k;\sigma,\xi)= { \left(1+\xi {k\over
  \sigma}\right)^{-1/\xi-1} \over \sum_{i=0}^\infty \left(1+\xi
  {i\over \sigma}\right)^{-1/\xi-1}}, \quad k=0,1,2,\ldots.
\end{align}
We call the probability mass function in (\ref{eq:zmrepa}) generalized Zipf distribution (GZD). In the case $\xi>0,$ the GZD is a Zipf--Mandelbrot distribution whose probability mass function is usually written in
the form  
$p(k)=(k+q)^{-s}/H_{s,q},$ for $k=0,1,2,\ldots,$ $s> 1$ and  $q>0$, where
$H_{s,q}$ is the Hurwitz-Zeta function (\cite{Mandelbrot1953}). The
GZD is of this form with $s=1+1/\xi$ and
$q=\sigma/\xi.$ When $q=1$ the distribution is a Zipf law which is
sometimes presented as the counterpart of the Pareto distribution
because its probability mass function, after a shift, can be written
in a homogeneous form (\cite{Arnold1985}).  Zipf-type
families of discrete distributions have been  fitted to various
data sets such as word frequencies (\cite{Booth1967}), city sizes
(\cite{Gabaix1999}),  company sizes (\cite{Axtell2001}) and numbers of
website hits (\cite{Clauset2009}). In the case $\xi=0,$ the generalized Zipf distribution is a geometric distribution
with probability of success $p=1-e^{-1/\sigma}$ (and so is the
discrete generalized Pareto distribution). 

We will show that the geometric, Poisson and
negative binomial probability mass  functions belong to
$\text{D-MDA}_{0}$, so the tail approximation \eqref{eq:approxZM}
makes sense in these cases. In addition, we will see that $p_X\in \text{D-MDA}_{\xi/(1+\xi)}$ for $\xi\geq 0$ implies $X\in\text{MDA}_{\xi}$ (under an additional condition in the case $\xi=0$).

Equations (\ref{eq:approxDGPD}) and (\ref{eq:approxZM}) motivate the approximation of the tails of some discrete distributions by a D-GPD and GZD, respectively. If $\xi\geq 0,$ then
\begin{align*}
\sup_{k=0,1,2\ldots} \left|{f_{\text{GPD}}(k;\sigma ,\xi) \over  q(k;\sigma ,\xi)} -1\right|\underset{\sigma\rightarrow \infty}{\longrightarrow} 0,
\end{align*}
for either $q = p_{\text{D-GPD}}$ or $q =p_{\text{GZD}}.$ One thus expects the GPD, D-GPD and GZD approximations to give similar results when the scale parameter $\sigma$ is large.

In Section 2, we justify theoretically the three approximations in the case $\xi>0$ by showing that if $p_X\in \text{MDA}_{\xi/(1+\xi)},$ then for any sequence $k_u\subset \mathbb N_0$ such that $\sup_u k_u/u<\infty$, 
 \begin{align*}
{P(X=k_u+u\mid X\geq u) \over q(k_u;\xi u,\xi)} \rightarrow 1,\quad k=0,1,2,\ldots,
 \end{align*}
 where $q\equiv f_{\text{GPD}},$ $p_{\text{D-GPD}}$ and $p_{\text{GZD}}$ (see Theorem \ref{thm:implications} and Proposition \ref{pro:convRatio}). A similar justification is provided in the case $\xi=0$ for the D-GPD and GZD approximations (Theorem \ref{thm:implications2}). We further present an important invariance property satisfied by the D-GPD (Proposition \ref{pro:invariance}). 

In Section 3, we simulate data from different discrete distributions and compare the ability of the three approximations for estimating the probability of regions far from the origin. The D-GPD and GZD approximations outperform the GPD when there are many tied observations, otherwise the results are similar. As opposed to the GZD, the D-GPD has a closed-form survival and probability mass function allowing for an exact likelihood based inference.

 In Section 4, we study a data set counting the occurrence of British words in a corpus and show that the three distributions are appropriate to describe the tails of these word frequencies. We also use the D-GPD and GZD to model the length of French words, the number of tornado outbreaks in the United States over $50$ years, and the number of births in the United States in the last $20$ years, illustrating how they offer potential methods for estimating the probability of rare events when the data are discrete.
 
% ===========================
% Section

\section{Theoretical Results} \label{sec:theory}

In this section we describe a number of properties of the
approximation procedures introduced in Section~\ref{par:intro}.

\begin{proposition}\label{pro:converg}
If $X\in \text{D-MDA}_{\xi},$ then there exists a positive sequence
$(a_u, \, u=1,2,\ldots)$ such that
\begin{align}\label{eq:convDiff}
\lim_{u\in \mathbb N_0, \, u\to\infty} \underset{k=0,1,2,\ldots}{\sup}\bigl|{P(X=u+k\mid  X\geq u)}-p_{\text{D-GPD}}(k;a_u, \xi)\bigr|= 0\,,
 \end{align}
where $p_{\text{D-GPD}}$ is defined in \eqref{e:p.d.gpd}. 
\end{proposition}
We remark that if $a_u\rightarrow \infty,$ then the result is less interesting because the two terms in (\ref{eq:convDiff}) converge to $0.$

\begin{proof}
By assumption, there exists a random variable $Y\in\text{MDA}_\xi$ for $\xi\geq 0$ and a positive
function $(\tilde a_u, \, u>0)$ such that $X\eid \lfloor Y\rfloor$ and the
sequence of functions $P\bigl( \tilde  a_u^{-1}(Y-u)\geq x\mid Y\geq
u\bigr)$, $x\geq 0$, converges uniformly, as $u\to\infty$, to the function $\bar 
F_{\text{GPD}}(x;\sigma,\xi)$, $x\geq 0$, for some $\sigma>0$ and
$\xi\geq 0$. For a positive integer $u$ we let $a_u= \tilde
a_u\sigma$. Then 
\begin{align*}
& \sup_{k=0,1,2,\ldots} \left|\,P(X=u+ k\mid X\geq u) - p_{\text{D-GPD}}(k;a_u,\xi)\,\right|\\
  = &\sup_{k=0,1,2,\ldots} \Bigl| P\bigl(\tilde  a_u^{-1} (Y-u)\geq
      \tilde  a_u^{-1} k\mid Y\geq u\bigr)-P\bigl(\tilde  a_u^{-1} (Y-u)\geq
      \tilde  a_u^{-1} (k+1)\mid Y\geq u\bigr) \\
& \hskip 0.5in-\bar F_{\text{GPD}}(k; a_u,\xi) +\bar F_{\text{GPD}}(k+1; a_u,\xi) \Bigr|\\
& \leq  2 \,\sup_{x\geq 0} \left| P\left(\tilde  a_u^{-1} (Y-u) \geq {x}\mid Y\geq u\right)-\bar F_{\text{GPD}}(x; \sigma,\xi)\right|
\rightarrow 0
\end{align*}
as $u\to\infty$ over the integers. 
\end{proof}

The following auxiliary lemma is elementary (as the sum can be
sandwiched between two integrals). 

\begin{lemma}\label{le:hzConv}
If $\xi>0$, then, as $u\to\infty$, 
$$ u^{1/\xi}  H_{1+1/\xi,u} \rightarrow \xi,$$
%G: $$ u^{1/\xi-1}  H_{1/\xi,u} \rightarrow \xi/(1-\xi),$$
where $H_{s,q}=\sum_{i=0}^\infty (q+i)^{-s}$ is the Hurwitz-Zeta function.
\end{lemma}
 
Recall that a positive and measurable function $f$ on $[1,\infty)$ is
regularly varying if there exists a  positive function $\ell$ such that
$$
\lim_{u\to\infty} {f(ux)\over f(u)}\rightarrow \ell(x),\quad x\geq 1.$$
In this case, there exists $\alpha\in \mathbb R$ such that $\ell(x)=x^{\alpha}$ and we write $f\in \text{RV}_\alpha$ (see e.g.\ \cite{Bingham1989}). If $f\in \text{RV}_{-\alpha}$ for $\alpha\geq 0,$ then
\begin{align}\label{eq:unifConvRV}
\lim_{u\to\infty} \sup_{x\in [1,b]}\left| { f(ux)\over f(u)}-x^{-\alpha}\right|\rightarrow 0,
\end{align}
for $b=\infty$ if $\alpha>0$, and for any $b<\infty$ if $\alpha=0.$ If $f\in
\text{RV}_{-\alpha}$ for $\alpha>0,$ then by Potter's bounds (see e.g.\ \cite{Resnick1987}) for any $\epsilon >0$ there
is $u_\epsilon\in (0,\infty)$ such that 
\begin{align}\label{eq:potterBound}
e^{-\epsilon} x^{-\alpha-\epsilon} \leq{ f(ux)\over f(u)} \leq  \,e^{\epsilon} x^{-\alpha+\epsilon},\quad x\geq 1,
\end{align}
for $u\geq u_\epsilon$. We say that $X$ is regularly varying if $\bar F_X\in
 \text{RV}_{-\alpha}$ for some $\alpha > 0,$ a necessary and sufficient condition for $X\in \text{MDA}_{1/\alpha}.$ 

The following result sheds some light on the approximation suggested
in \eqref{eq:approxZM}. 

\begin{theorem}\label{thm:implications} 
If $p_X\in \text{D-MDA}_{\xi/(1+\xi)}$ for $\xi>0,$ then $X\in
\text{MDA}_{\xi}$ and for any sequence of nonnegative integers
$(k_u, \, u=1,2,\ldots)$ such that $\sup_u k_u/u<\infty$, 
 \begin{align}\label{eq:convpZMRatio}
\lim_{u\in \mathbb N_0, \, u\to\infty} {P(X=k_u+u\mid X\geq u) \over
   p_{\text{GZD}}(k_u;\xi u,\xi)} = 1.
 \end{align}
\end{theorem}

\begin{proof}
% TODO: \deleted[remark={TODO question: Gena, is the first part clear? I don't understand why}]{We have already mentioned the first statement of the theorem}.

By assumption, there exists a survival function $\bar F$ such that $\bar F(k)=c\, p_X(k)$ for $c>0,$ $k$ large enough and $\bar F\in \text{RV}_{-1/\xi-1}.$ The last condition is equivalent to $\bar F(\lfloor \cdot \rfloor) \in \text{RV}_{-1/\xi-1}$ (\cite{Shimura2012}). Therefore, 
\begin{align*}
{\pr(X\geq u x)\over \pr(X\geq u)} &= {\sum_{i=\lceil u x\rceil}^\infty p_X(i) \over \sum_{i=u}^\infty p_X(i)} 
 ={ \int_{ux}^\infty \bar F(\lfloor y \rfloor )dy -\int_{ux}^{\lceil ux \rceil} \bar F(\lfloor y \rfloor )dy \over \int_{u}^{\infty }  \bar F(\lfloor y \rfloor )dy } \\
 & ={ u \int_{x}^\infty \bar F(\lfloor u z \rfloor )/\bar F(\lfloor u \rfloor )dz - (\lceil u x\rceil-ux ) \bar F(\lfloor u x \rfloor )/\bar F(\lfloor u \rfloor ) \over u \int_{1}^{\infty }  \bar F(\lfloor u z \rfloor )/\bar F(\lfloor u \rfloor )dz } \rightarrow x^{-1/\xi},\quad x\geq 1,
 \end{align*}
applying (\ref{eq:potterBound}) and dominated convergence. Thus, $\bar F_X\in\text{RV}_{-1/\xi}$ and $X\in\text{MDA}_{\xi}.$

For the second part of the theorem, we have 
$$ 
{P(X=k_u+u\mid X\geq u) \over p_{\text{GZD}}(k_u;\xi u,\xi)} 
= \frac{\bar F(u+k_u)/\bar F(u)}{(1+k_u/u)^{-1/\xi-1}}
\frac{\sum_{i=0}^\infty
  (1+i/u)^{-1/\xi-1}}{\sum_{i=0}^\infty \bar F(u+i)/\bar F(u)}.
$$ 
By the uniform convergence \eqref{eq:unifConvRV} and the the fact that
$k_u$ grows at most linearly fast, we conclude that 
$${\bar F(u+k_u)/\bar F(u) \over (1+k_u/u)^{-1/\xi}} \rightarrow 1,$$
as $u\to\infty$ over the integers. Second, Lemma \ref{le:hzConv} yields
\begin{align*}
u^{-1} {\sum_{i=0}^\infty (1+i/u)^{-1/\xi-1} } \rightarrow \xi.
\end{align*}
Third, it follows from (\ref{eq:potterBound}) that for $\epsilon\in
(0,1/\xi),$ there exists $u_\epsilon>0$  such that for $u\geq
u_\epsilon$, 
\begin{align*}
u^{-1} \sum_{i=0}^\infty \bar F(u+i)/\bar F(u)
&\leq u^{-1}  e^{\epsilon} \sum_{i=0}^\infty \left(1+{i\over u}\right)^{-1-1/\xi+\epsilon} \rightarrow\;  {\xi  e^\epsilon \over 1-\xi \epsilon},
\end{align*}
using again Lemma~\ref{le:hzConv}. A lower bound is found in the same manner and we let $\epsilon\rightarrow 0$ to conclude.
\end{proof}

We now present a tail equivalence property between the probability mass and density functions of the GZD, D-GPD and GPD. A direct consequence is that the denominator $p_{\text{GZD}}$ in
(\ref{eq:convpZMRatio}) can be replaced either  by $p_{\text{D-GPD}}$
or by $f_{\text{GPD}}.$ 

\begin{proposition}\label{pro:convRatio}
If $\xi\geq 0,$ then%
\begin{align*}%\label{eq:convRatio3} 
\lim_{\sigma\to\infty} \sup_{k=0,1,2,\ldots}\left| {p_{\text{D-GPD}}(k;
  \sigma, \xi)\over p_{\text{GZD}}(k; \sigma ,\xi)}  -1 \right| =
\lim_{\sigma\to\infty} 
\sup_{k=0,1,2,\ldots} \left| {p_{\text{D-GPD}}(k; \sigma, \xi)\over
  f_{\text{GPD}}(k; \sigma,\xi)}-1 \right|= 0. 
\end{align*}
\end{proposition}
\begin{proof}
Suppose first that $\xi>0$. Then 
\begin{align*}
{p_{\text{D-GPD}}(k;\sigma, \xi)\over f_{\text{GPD}}(k;\sigma, \xi)} &= 
{ \left(1+\xi{ k\over \sigma}\right)^{-1/\xi}-\left(1+\xi{k+1\over \sigma}\right)^{-1/\xi} \over {1\over \sigma } \left(1+\xi {k\over \sigma}\right)^{-1/\xi-1}} \\
&=  \left\{1-\left(1+{\xi\over \sigma+\xi k}\right)^{-1/\xi}\right\} (\sigma+\xi k)\rightarrow 1,\nonumber
\end{align*}
uniformly in $k=0,1,2,\ldots$. Furthermore, 
\begin{align*}
  \sup_{k=0,1,2,\ldots} {f_{\text{GPD}}(k;\sigma,\xi)\over  p_{\text{GZD}}(k;\sigma,\xi) }    &= 
   \sigma^{-1} \,\sum_{i=0}^\infty (1+\xi i/\sigma)^{-1/\xi-1} \rightarrow 1
 \end{align*}
by Lemma~\ref{le:hzConv}. In the case $\xi=0,$ 
$$
{p_{\text{D-GPD}}(k;\sigma, 0)/f_{\text{GPD}}(k;\sigma, 0)}= {p_{\text{GZD}}(k;\sigma, 0)/f_{\text{GPD}}(k;\sigma, 0)}= 
\sigma
(1-e^{-1/\sigma})\rightarrow 1
$$
as $\sigma\to\infty$. 
\end{proof}

Next we extend Theorem~\ref{thm:implications} to the case
$\xi=0$. Recall that a distribution $F$ is in $\text{MDA}_0$ if and
only if the survival function has a representation 
\begin{equation} \label{e:d.0}
\bar F(x)=c(x)\exp\left\{-\int_0^x {1\over a(y)} dy\right\},\quad  -\infty<x<x_F,
\end{equation}
where $c(\cdot)$ is a positive function with $c(x)\to c>0$ as $x\uparrow 
x_F$, and  $a(\cdot)$ is a positive, differentiable function $a(\cdot)$
 with  $\lim_{x\uparrow x_F} a'(x)=0$. If $c(x)=c$ on $(d,x_F)$ for
 some $d<x_F$, then we say that  the
 distribution $F$ satisfies the  von Mises condition. The function
 $a(\cdot)$ in \eqref{e:d.0} is sometimes called the  auxiliary
 function. Note, however, that it is only uniquely defined (on
 $(d,x_F)$) under the  von Mises condition; see  
\cite{Embrechts2013}.  Recall
that in the sequel we only consider the case of unbounded support,
i.e. $x_F=\infty$.

\begin{theorem}\label{thm:implications2}
Suppose that $p_X\in \text{D-MDA}_0$ and, moreover, that a distribution $F$
such that $p_X(k)=\bar F(k)$ has the property that an auxiliary function of $\bar F$
satisfies 
$\lim_{x\rightarrow \infty} a(x)= \sigma\in [0,\infty]$. Then 
 \begin{align}\label{eq:convpZMRatio2}
\lim_{u\in \mathbb N_0, \, u\to\infty}P(X=k+u\mid X\geq u)  =
   p_{\text{Ge}}(k;\sigma) ,\quad k=0,1,2,\ldots\,,
 \end{align}
 where $p_{\text{Ge}}(k;\sigma) = (1-e^{-1/\sigma})\,  e^{-k/\sigma}$ 
   is the probability mass function of a geometric
 distribution if $0<\sigma<\infty$, and  $p_{\text{Ge}}(k;\infty) =p_{\text{Ge}}(k;0)=0$. Furthermore, if $\sigma\in [0,\infty),$ then $X\in \text{D-MDA}_0.$ 
\end{theorem}
\begin{proof}
Note that for large integers $u$, 
\begin{equation} \label{e:for.ge}
P(X=k+u\mid X\geq u)  = \frac{\bar F(k+u)/\bar
  F(u)}{\sum_{i=0}^\infty \bar F(i+u)/\bar  F(u)}\,.
\end{equation}
We have for every $i=0,1,2,\ldots$, 
$$
 \bar F(i+u)/\bar  F(u) 
  = \frac{c(i+u)}{c(u)}\exp\left\{-\int_0^i{1/ a(u+y)}dy\right\}  \rightarrow e^{-i/\sigma}
$$ 
as $u\to\infty$. If $0<\sigma<\infty$, then the dominated convergence
theorem gives us 
$$ 
\sum_{i=0}^\infty \bar F(i+u)/\bar  F(u) \to \sum_{i=0}^\infty
e^{-i/\sigma} = 1/(1-e^{-1/\sigma})\,,
$$ 
and \eqref{eq:convpZMRatio2} follows. If $\sigma=\infty$, 
$$ 
\sum_{i=0}^\infty \bar F(i+u)/\bar  F(u) \to\infty
$$
by Fatou's lemma, and \eqref{eq:convpZMRatio2}, once again,
follows. If $\sigma=0$, the claim follows from the fact that the
denominator in \eqref{e:for.ge} cannot be smaller than 1. 

For the second part of the proposition, it follows from
$$
p_X(n) = c(n)\exp\left\{-\int_0^n {1\over a(y)} dy\right\}
$$ 
for all $n$ and $a(y)\to\sigma \in [0,\infty)$ that
$$
\lim_{n\to\infty} \frac{p_X(n)}{P(X\geq n)} =
1-e^{-1/\sigma}\in (0,\infty)\,,
$$
which immediately implies that $X\in \text{D-MDA}_0$ as well. 
\end{proof}

To summarize, for a discrete random variable $X$, 
the conditions $X\in \text{MDA},$ $X\in \text{D-MDA}$
and $p_X\in \text{D-MDA}$ are  related to each other as follows. 
 If $\xi\geq 0,$ then $X\in  \text{MDA}_{\xi}$ if and only if $ X\in
 \text{D-MDA}_{\xi}  \mbox{ and } X  \mbox{ is long-tailed}.$ If $\xi> 0,$ then $p_X\in \text{D-MDA}_{\xi/(1+\xi)}  $ implies $X\in
 \text{D-MDA}_{\xi}$; the same implication holds in the case $\xi=0$ if the auxiliary function of the extension of $p_X$ satisfies $a(x)\rightarrow\sigma\in  (0,\infty)$ as $x\rightarrow\infty.$

The condition $p_X\in \text{D-MDA}$ is satisfied, among others, by the
Zipf--Mandelbrot, geometric, Poisson and negative binomial
distributions, as shown in the next example. 

\begin{example}\label{ex:dmda}
The probability mass function of a Zipf--Mandelbrot distribution with
parameters $s>1$ and $q>0$ is in $\text{D-MDA}_{1/s}$ because it
is regularly varying of order  $-s$.  

 The probability mass function of a geometric distribution belongs to
 $\text{D-MDA}_0$ as it coincides up  to a constant with the survival
 function of an exponential distribution.  The latter distribution
 clearly satisfies the von Mises condition and thus is a member of
 $\text{MDA}_0$. The auxiliary function  is, in fact, equal
 (eventually) to $1/\lambda,$ where $\lambda$ is  the rate of the
 exponential distribution. 

 The probability mass function $p$ of a Poisson distribution with rate
 $\lambda>0$ coincides on  $k=0,1,2,\ldots$ with the function 
$$ g(x)={\lambda^x e^{-\lambda} \over \Gamma(x+1)},$$ 
 a continuous function on $\mathbb R_+$ satisfying $\lim_{x\rightarrow
   \infty}  g(x)= 0$. Moreover, 
$$
\frac{d}{dx} \log g(x) = -\psi_0(x+1) + \log \lambda\,,
$$
where $\psi_0$ is the polygamma function of order 0. Since
$\psi_0(x)\to \infty$ as $x\to\infty$, we see that 
$ g^\prime(x)<0$ for $x$ sufficiently large. Therefore, $ \bar F_Y(x) =
g(x)/ g(d)$ is a survival function on  $[d,\infty)$ for some $d\geq
0.$ Furthermore, 
$$
\frac{d}{dx} \left( -\frac{1}{g^\prime(x)} \right)= -\frac{\psi_1(x+1)}{(\psi_0(x+1) - \log \lambda)^2}\,,
$$
where $\psi_1=\psi_0^\prime$ is is the polygamma function of order
1. Since $\psi_1(x)\to 0$ as $x\to\infty$, we conclude that $F$
satisfies the  von Mises condition, with  the auxiliary function
$a(x)=(\psi_0(x+1) - \log \lambda)^{-1}\to 0$ as
$x\to\infty$. Therefore, the Poisson probability mass function is in
$\text{D-MDA}_0$. 
 
Similarly, the probability mass function of the negative binomial
distribution with probability of success  $p\in (0,1)$ and number of
successes $r>0$ is also in  $\text{D-MDA}_0$ because it coincides on
$\{ 0,1,2,\ldots\}$ with the function 
$$
g(x)={p^r\over \Gamma(r)} {\Gamma(x+r)\over \Gamma(x+1)}(1-p)^x\,, 
$$
a continuous function on $\mathbb R_+$. It is simple to check that
$\lim_{x\rightarrow \infty} g(x) = 0,$ and $g'(x)<0$ for $x$ large
enough, so that $ \bar F_Y(x) =
g(x)/ g(d)$ is a survival function on  $[d,\infty)$ for some $d\geq
0$. Furthermore, $g(x) \sim cx^{r-1}(1-p)^x$ for large $x$, where
$c$ is a positive constant. Therefore, $\bar F_Y$ is of the form
\eqref{e:d.0} with the auxiliary function 
$$
a(x) =\frac{1}{-\log (1-p) - (r-1)/x}, \ \ \text{$x$ large,}
$$
and so it converges to $-1/\log (1-p)$ as $x\to\infty$. 
\end{example}

We conclude this section with a discussion designed to provide 
some intuition on how the approximation methods suggested above
differ, assuming that both apply in a given situation. 
 First of all, one would  expect the D-GPD  and GZD approximations to perform
similarly when $\xi$ is close to  zero because both approximating 
distributions coincide with a geometric distribution when $\xi=0.$
Second, Proposition~\ref{pro:convRatio}  suggests that, regardless of
the underlying justification, if one uses either $p_{\text{D-GPD}}(k;
  \sigma, \xi)$, $f_{\text{GPD}}(k; \sigma,\xi)$ or $p_{\text{GZD}}(k;
  \sigma ,\xi )$ as an approximation to $P(X-u= k\mid
  X\geq u)$, one should not expect major differences as long as one
  chooses the scale  parameter $\sigma$ to be large. This would always
  be the case if $X$ is   long-tailed, since the scale 
   parameter is chosen to be proportional to the normalization sequence 
  $a_u$ defined in (\ref{eq:convG}), which grows to infinity if and only if $X$ is
  long-tailed. 

When using a continuous distribution, such as the generalized Pareto
distribution,  to approximate the probabilities related to 
a discrete distribution, it is also common to use a
``continuity correction'' and shift the argument in the continuous
approximation by some $\delta\in [0,1)$. In our situation this amounts
to replacing $f_{\text{GPD}}(k; \sigma,\xi)$
by $f_{\text{GPD}}(k+\delta; \sigma,\xi)$, some $\delta\in [0,1)$. When $\xi>0,$ it
is elementary to check that, as $\sigma\to\infty$,  
\begin{align*}
{p_{\text{D-GPD}(\sigma,\xi)}(k)\over
  f_{\text{GPD}(\sigma,\xi)}(k+\delta)}  = 
1+{(1+\xi)(2\delta-1) \over 2\sigma } +O(\sigma^{-2}) 
\end{align*}
for every $k=0,1,2,\ldots$. Therefore, the approximations by
$p_{\text{D-GPD}(\sigma,\xi)}(k)$ and $f_{\text{GPD}}(k+\delta;
\sigma,\xi)$ with large $\sigma$ are most similar when
$\delta=1/2$, a property that will be  illustrated in the empirical
part. Similarly, in the case $\xi=0,$ as $\sigma\to\infty$, 
\begin{align*}
{p_{\text{D-GPD}(\sigma,\xi)}(k)\over
  f_{\text{GPD}(\sigma,\xi)}(k+\delta)}=  \sigma e^{\delta/\sigma} (
  1-e^{-1/\sigma})=  1+\frac{2 \delta-1}{2\sigma}+O(\sigma^{-2})
\end{align*}
for every $k=0,1,2,\ldots$, 
and the fastest convergence to the unity is again found when $\delta=1/2$. 

The final result of this section accomplishes two things. Its first
part is related to the well-known invariance property of the
generalized Pareto distribution: its residual lifetime 
is again generalized Pareto distributed. More precisely, if
$Y$ has the generalized Pareto distribution with scale parameter
$\sigma$ and shape parameter $\xi\geq 0$, then the exceedance $Y-u$ has,
given $Y\geq u$, the generalized Pareto distribution with scale
parameter $\sigma+\xi u$ and shape parameter $\xi$, for all  $u\geq
0$. This invariance property is important  when approximating the
exceedance distribution using generalized Pareto distributions
because changing the threshold does not alter the distributional
assumptions used in the approximation.  It is easily checked that the
D-GPD has an analogous property. Moreover, the D-GPD also exhibits the property that discretizing a GPD using different types of rounding does not affect the fact that a D-GPD is obtained, and the shape parameter remains invariant. 

% TODO: \deleted{This property is combined in the following proposition with an also easily checked fact that an exact nature of discretizing a generalized Pareto distribution does not affect the fact that a D-GPD is obtained, even though the scale parameter (but not the shape parameter) may change.}

\begin{proposition}\label{pro:invariance}
Let $Y$ follow a GPD with scale
parameter $\sigma>0$ and shape parameter $\xi\geq 0$. Let $0<h\leq 1$. 
If $X=\lfloor \lambda Y +1-h\rfloor$, 
then for any integer $u \geq 1-h$, the distribution of $X-u\mid X\geq
u$ is a D-GPD with shape parameter $\lambda \sigma
+\xi(u+h-1)$ and scale parameter~$\xi$. 
 \end{proposition}

% ===========================
% Section

\section{Empirical Results}

We now assess the performance of the generalized Pareto distribution (GPD), the discrete generalized Pareto distribution (D-GPD) and the generalized Zipf distribution (GZD) approximations for estimating extreme quantiles from several simulated and real data sets, starting with a simple example to illustrate the methods.

\begin{figure}[h]
    \centering
    \begin{subfigure}{0.49\textwidth}
        \includegraphics[width=\textwidth]{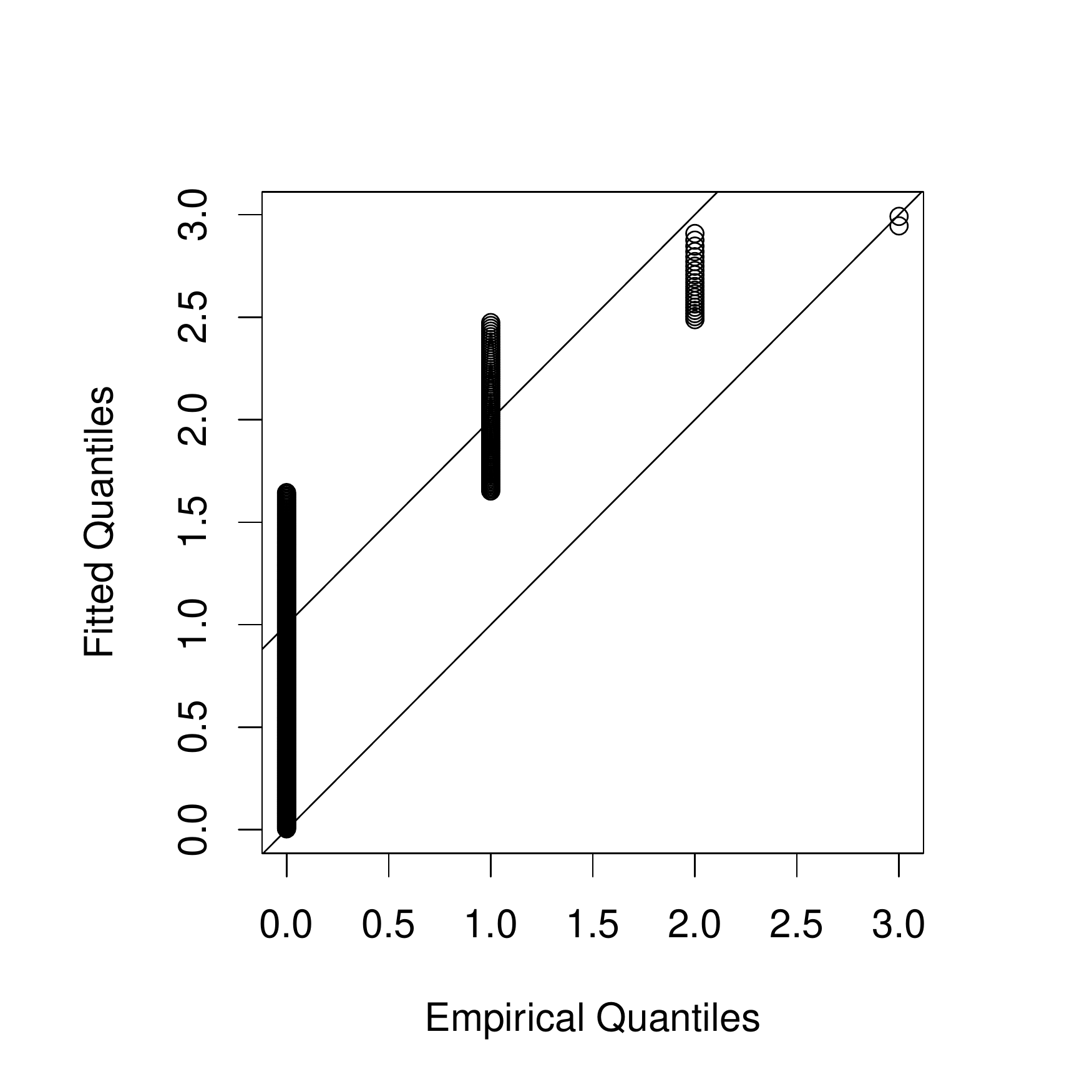}
    \end{subfigure}
        \begin{subfigure}{0.49\textwidth}
        \includegraphics[width=\textwidth]{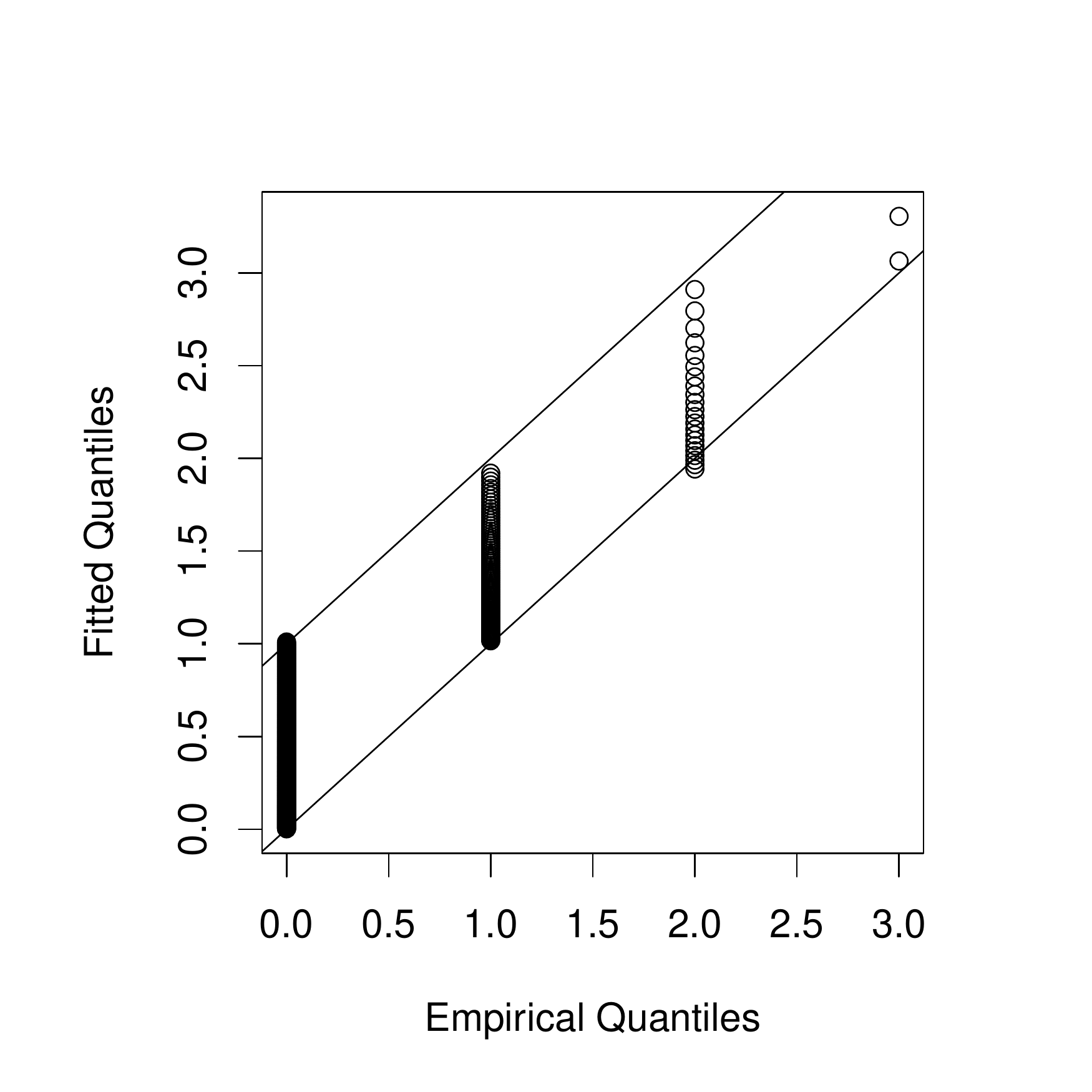}
    \end{subfigure}
            \caption{Quantile-quantile plots for the fit of a GPD (left) and D-GPD (right) to $X-3\mid X\geq 3$, where $X$ is Poisson distributed with rate $1$ and the sample size is $5000.$ A good fit occurs on the right because the points are contained between the two lines, but not on the left. 
    }\label{fig:qqplotPoisson}
\end{figure}

\subsection{Simulated Data}

Let $X$ be Poisson distributed with rate $\lambda=1$ and consider an i.i.d.\ sample of size $n=5000.$ We are interested in inferring the distribution of the exceedances $X-u\mid X\geq u$ above a large threshold $u$, say $u=3$ which is the $95$th empirical percentile of the data in this case. Since $X\in \text{D-MDA}_0,$ we can approximate the distribution of $X-u\mid X\geq u$ by a D-GPD as explained in Section~\ref{par:intro}. Moreover, $p_X,$ the probability mass function of $X,$ also belongs to $\text{D-MDA}_0$ and satisfies the additional assumption in Theorem \ref{thm:implications2} as shown in Example~\ref{ex:dmda}, motivating the D-GPD and GZD approximations. However, $X$ is not in $\text{MDA}$ and thus the GPD approximation does not necessarily apply. In order to compare these three approximations, we fit a GPD, D-GPD and GZD to the observations above $u=3,$ estimating their two parameters $\sigma$ and $\xi$ by maximum likelihood. Figure~\ref{fig:qqplotPoisson} compares the fitted quantiles of the GPD (left) and D-GPD (right) relative to the empirical quantiles. Since the data are discrete, we applied a slightly different graphical method than the standard quantile-quantile (QQ) plots: $F_{\text{GPD}}^{-1}\{i/(n+1);\sigma,\xi\},$ for $i=1,\ldots,n,$ were plotted against empirical quantiles, where $(\sigma, \xi)$ are the estimated parameters. A good fit occurs when the points accumulate between the two diagonal lines in Figure~\ref{fig:qqplotPoisson} and touch the bottom line, as would be the case if the quantiles of a continuous random variable $Y$ were plotted against those of $\lfloor Y\rfloor.$ One clearly sees that the D-GPD fits well the observations compared to the GPD which delivers a poor fit. We mention that the D-GPD and GZD approximations produce very similar estimates in this case; the QQ-plot of the GZD (not displayed here) looks visually identical to the one of the D-GPD. The estimated scale parameters $\sigma$ for the GPD, D-GPD and GZD are $1.91,$ $0.71$ and $0.69$ respectively. Increasing $\lambda$ would increase these estimates and render the three methods indistinguishable as expected from Proposition~\ref{pro:convRatio}.

%This graphical method for model diagnostic differs slightly from standard quantile-quantile (QQ) plots because the variable on the $x$-axis only takes discrete values.

We now compare the performances of the GPD, D-GPD and GZD approximations in estimating the probability of a rare event in the following simulated case. Let 
\begin{align}\label{eq:defXY}
Y\sim \text{IG}(2,1), \quad X=\lfloor Y\rfloor,\end{align}
where $\text{IG}(\alpha,\beta)$ denotes an inverse-gamma distribution with probability density function $f(x)=\Gamma(\alpha)^{-1} \beta^\alpha x^{-\alpha-1} \exp(-\beta/x),$ $x>0$. We repeat $500$ times the experiment described below. An i.i.d.\ sample of size $8000$ is drawn from the distribution of $X$. From these observations, the goal is to estimate the probability of the extreme region
\begin{align}\label{eq:defpe}
p_e=P(X\geq \lfloor q_e\rfloor),
\end{align}
where $q_e$ is the $99.99$ percentile of $Y,$ i.e., the value exceeded once every $10\, 000$ times on average. The strategy pursued is to select an integer threshold $u$ as the $95$th empirical percentile of the sample, fit a parametric distribution to the exceedances $X-u\mid X\geq u,$ and use it to extrapolate the tail and estimate $p_e$. It holds $X\in \text{D-MDA}_{1/2}$ and $X\in \text{MDA}_{1/2},$ which motivates the choice of a GPD and D-GPD as seen in Section~\ref{par:intro}. We implement two variants of the GPD approximation: the first has no continuity correction and the second shifts the observations by $-\delta=-\frac{1}{2}.$ It is not immediate if the probability mass function of $X$ is in D-MDA and we thus apply the GZD approximation heuristically. 

As a benchmark, we will also estimate $p_e$ from a sample of the continuous variable $Y$ (as opposed to its discretization $X$). In this context, the GPD approximation is motivated by the fact that $Y\in \text{MDA}_{1/2},$ and we thus fit a GPD to $Y-u\mid Y\geq u.$

\begin{figure}    
\centering    
     \begin{subfigure}[b]{0.495\textwidth}
        \includegraphics[width=\textwidth,trim=0.8cm 0.5cm 0.8cm 0.5cm]{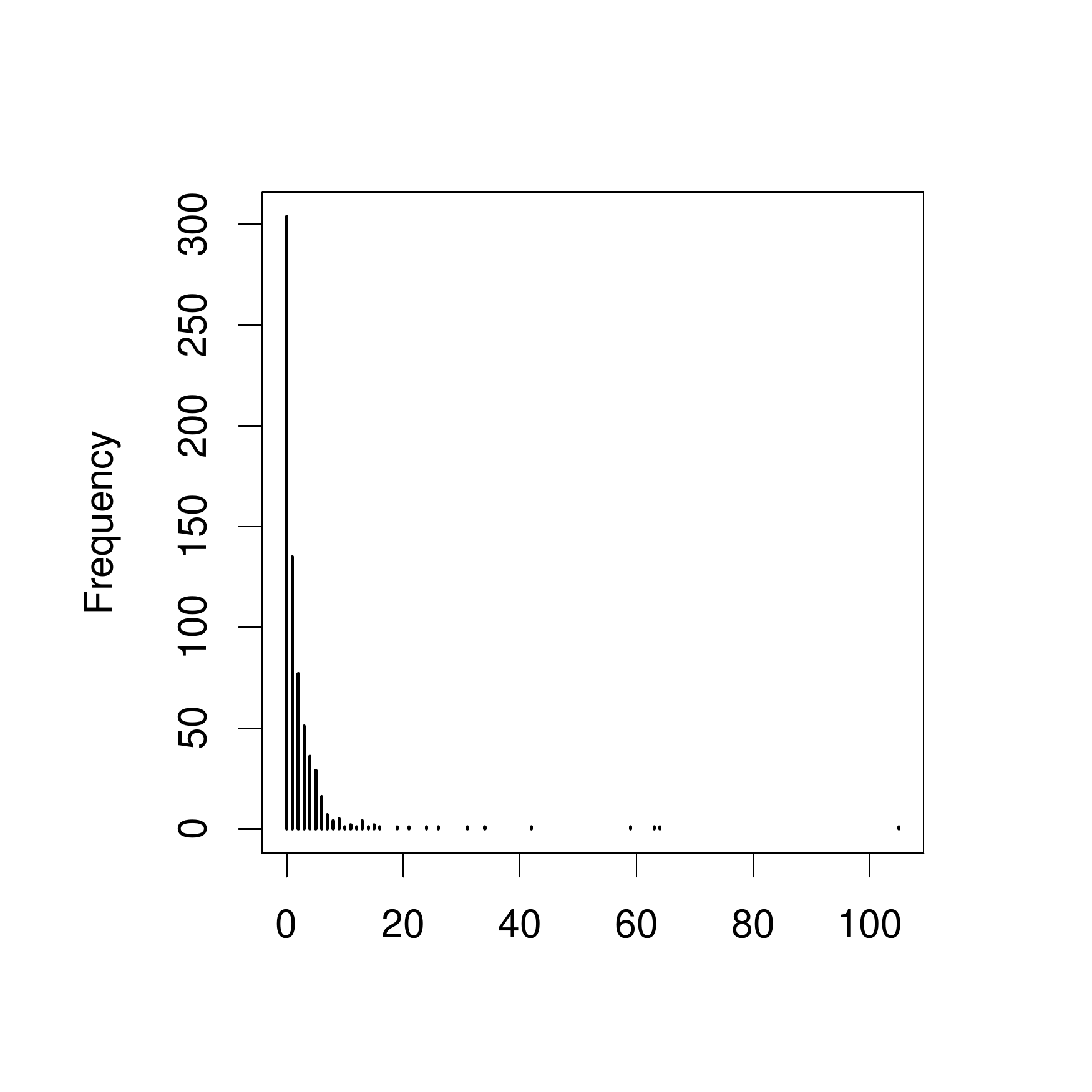}
%        \caption{Experiment $1$}
    \end{subfigure}
      \begin{subfigure}[b]{0.495\textwidth} %<left> <lower> <right> <upper>}
        \includegraphics[width=\textwidth,trim={0.8cm 0.5cm 0.8cm 0.5cm}]{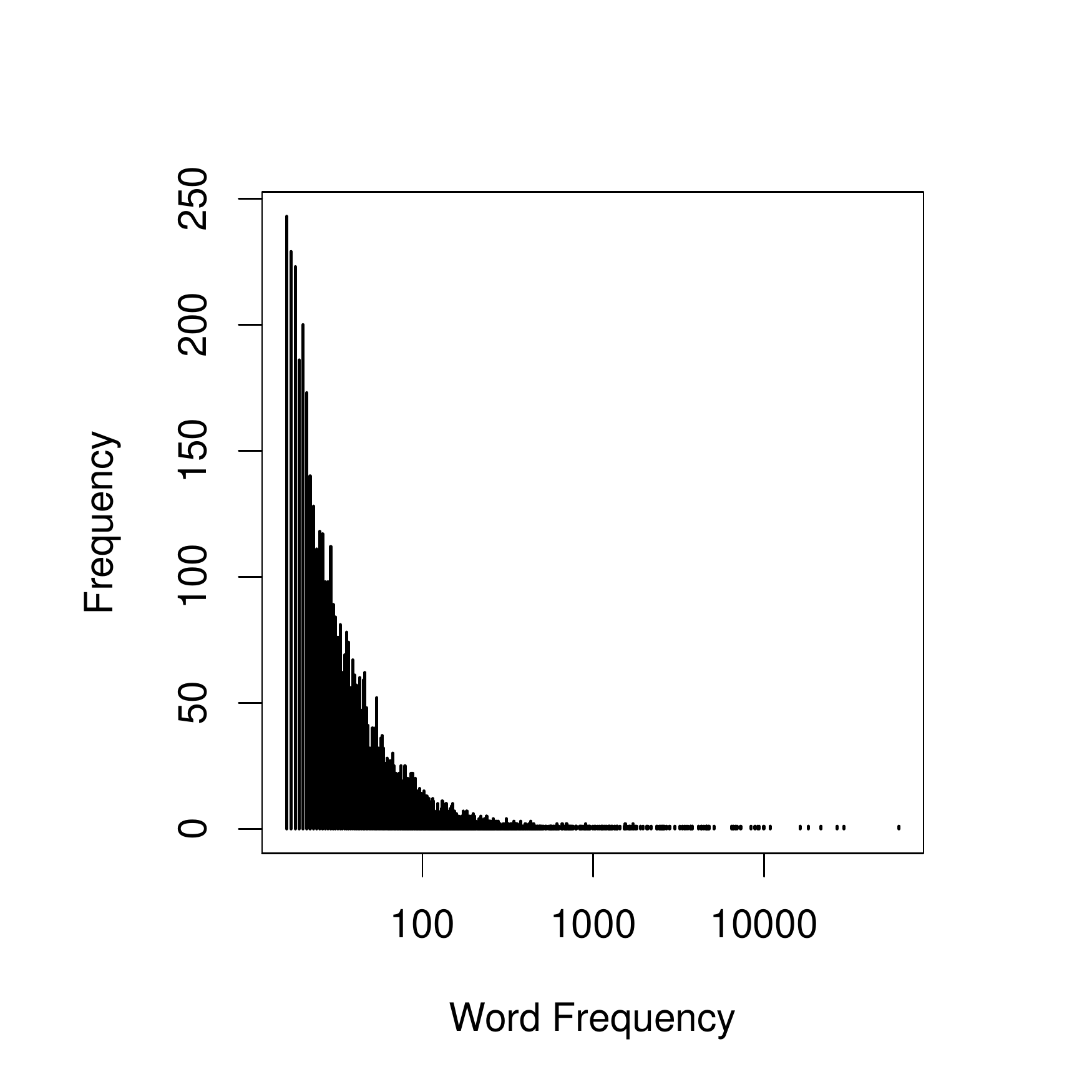}
    \end{subfigure}
    \caption{On the left: frequency plot of a sample of $X-2\mid X\geq 2$ of size $701$ for $X$ defined in (\ref{eq:defXY}). On the right: frequency plot of the $5588$ most frequent words in a British corpus ($x$ axis on log-scale).
    }
         \label{fig:tablePlotSim}
\end{figure}

\begin{table}
\small
\center
\begin{tabular}{l|lllll}
%$ P( X \geq 70  ) \times 10^ 3 $
& $ p_e \times 10^ 3 $&  &  $ \xi $  & $ \sigma$   \\ 
\hline\hline
&  mean (cov, len)   & true length    & mean (cov, len) &  mean (len) \\ 
 truth  & $ \textbf{0.10}  $ & &  $0.5$ &  \\ 
 \hline $Y-u\mid Y\geq u$  &  & $ $ & \\  
 GPD & $ 0.10 \, ( 87 \%, \,  0.16 ) $ & $ 0.16 $ & $ 0.49 \, ( 95 \%, \,  0.22 ) $ & $ 1.19 \, ( 0.30 ) $ \\ 
 \hline  $X-u\mid X\geq u$ &  & $ $ & \\  
 D-GPD & $ \textbf{0.10} \, ( 86 \%, \,  0.17 ) $ & $ 0.16 $ & $ 0.49 \, ( 95 \%, \,  0.23 ) $ & $ 1.19 \, ( 0.33 ) $ \\ 
 GZD & $ \textbf{0.11} \, ( 88 \%, \,  0.17 ) $ & $ 0.17 $ & $ 0.50 \, ( 95 \%, \,  0.24 ) $ & $ 1.39 \, ( 0.29 ) $ \\ 
 $ \text{GPD}_{\delta=\frac{1}{2}} $ & $ 0.04 \, ( 20 \%, \,  0.07 ) $ & $ 0.08 $ & $ 0.37 \, ( 22 \%, \,  0.18 ) $ & $ 1.43 \, ( 0.32 ) $ \\ 
 $\text{GPD}_{\delta=0}$ & $ 7.93 \, ( 83 \%, \,  23.97 ) $ & $ 11.25 $ & $ 8.27 \, ( 0 \%, \,  1.24 ) $ & $ 0.00 \, ( 0.00 ) $ \\
 \end{tabular}
\caption{Performance of several methods in estimating the probability $p_e$ of a rare event defined in (\ref{eq:defpe}). A generalized Pareto distribution (GPD) is fitted to the observations of $Y$ exceeding $u=2,$ and the following distributions are fitted to exceedances of $X=\lfloor Y\rfloor$: a discrete generalized Pareto distribution (D-GPD), a generalized Zipf distribution (GZD) and a GPD with continuity correction $\delta.$ The table displays average maximum likelihood estimators for $p_e,$ $\xi$ and $\sigma$ accross $500$ experiments. Coverage (cov) and average length (len) of $90\%$ confidence intervals are shown between brackets. In each experiment, about $700$ observations exceeded the threshold.}
  \label{tab:exp1Init}
 \end{table}

A frequency plot of a sample of $X-u\mid X\geq u$ is displayed in Figure~\ref{fig:tablePlotSim} on the left-hand side. For each model, we compute maximum likelihood estimators for $\sigma$ and $\xi$ by performing a two dimensional maximization using the function \texttt{optim} of R (\cite{R2015}) with starting values $(1,1).$ We then compute $p_e$ and approximate $90\%$ confidence intervals under asymptotic normality of the estimators. Table~\ref{tab:exp1Init} displays: the average parameters $p_e,$ $\xi$ and $\sigma$ over the $500$ experiments, the average length of the confidence intervals, the true length and their coverage. True length is the length the intervals should have had to contain the estimates across the $500$ experiments $90\%$ of the time. Coverage indicates the proportion of time the truth lies in the confidence interval. 

It appears that the D-GPD and GZD approximations accurately estimate $p_e$ from the discretized data with a coverage close to the correct one of $90\%,$ and that their performance is good relative to the situation of full information where the continuous data are available. On the other hand, the $\text{GPD}$ approximation with $\delta=\frac{1}{2}$ is inaccurate and the GPD with $\delta=0$ performs very poorly (misleadingly, the latter has a larger likelihood at the maximum likelihood estimate than the former). This illustrates why the D-GPD or GZD approximations should be preferred to the GPD approximation in situations similar to this simulated case. 
 
We know that the D-GPD and GZD coincide with a geometric distribution when $\xi=0$ and that they are asymptotically equivalent when $\xi>0$ and $\sigma\rightarrow \infty$ (Proposition~\ref{pro:convRatio}). Although $\sigma$ is not particularly large in the above simulated case, the D-GPD and GZD approximations deliver very similar results. We point out that they outperform the Poisson and negative binomial distributions which would estimate $p_e$ very poorly. 

We remark few differences between the GPD fitted to exceedances of $Y$ and the D-GPD fitted to exceedances of its discretization $X$, except that the former method yields slightly shorter confidence intervals. As a consequence of the invariance property in Proposition~\ref{pro:invariance}, the two methods would still give similar estimates of $\xi$ if $X$ would be defined by other forms of rounding than $X=\lfloor Y\rfloor,$ such as using rounding or ceiling function.

In conclusion, we showed that the D-GPD and GZD approximations were able to estimate the probability of rare events in this example. These findings are supported by two complementary simulated cases with $Y\in \text{MDA}_\xi$ for $\xi=0$ and $\xi<0$ (\cite{thesisHitz2017}, Chapter~2). The D-GPD provides a more efficient method because its probability mass, survival and quantile functions are closed-form, but both distributions are useful to describe extreme values of discrete random variables, which we will further illustrate on real data sets.

\subsection{Word Frequencies and Word Lengths}\label{subsec:wf}

\begin{figure}
\center
    \begin{subfigure}[b]{0.45\textwidth}
        \includegraphics[width=\textwidth, trim= 1cm 0.5cm 1.5cm 1.5cm]{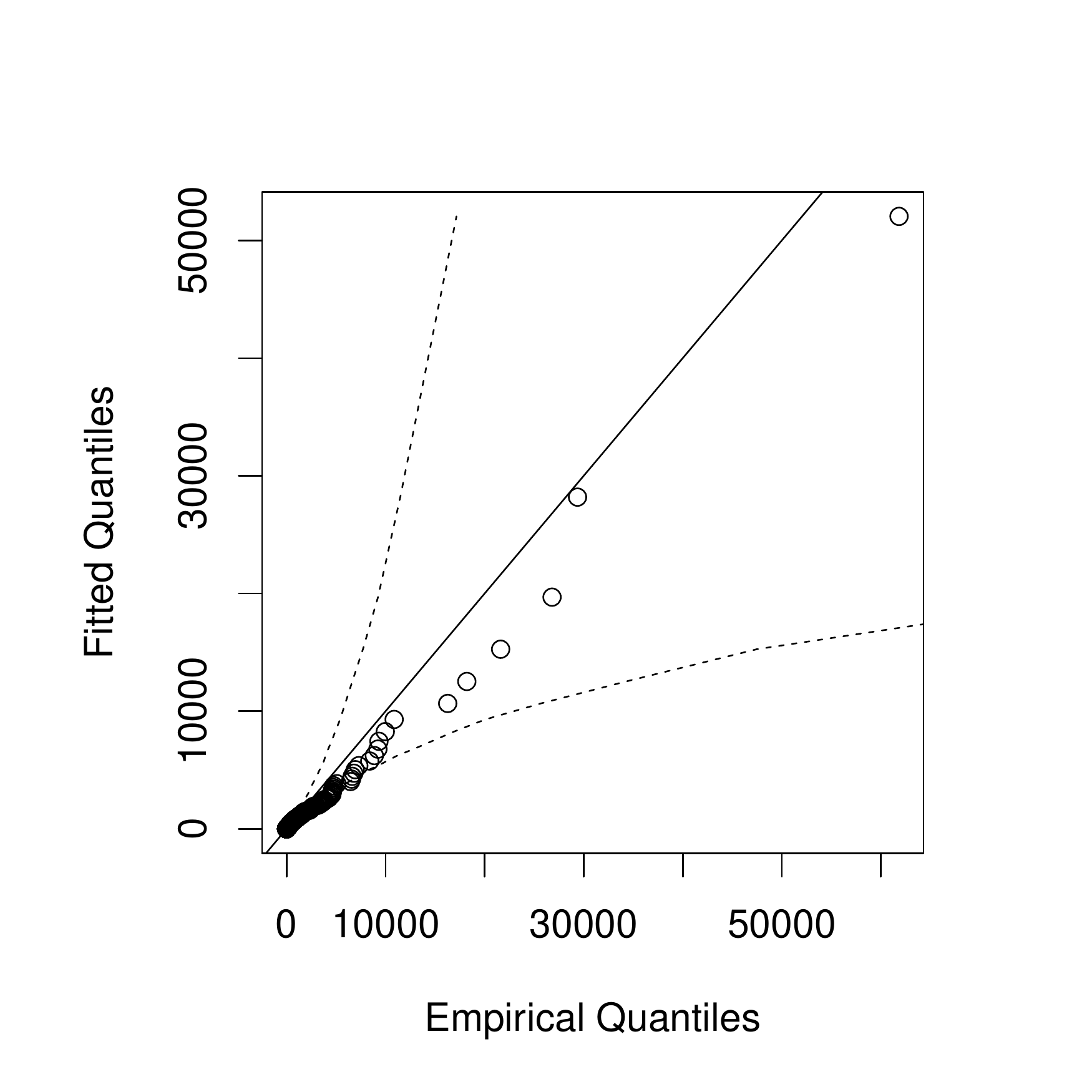}
    \end{subfigure}
     \begin{subfigure}[b]{0.45\textwidth}
        \includegraphics[width=\textwidth, trim= 1.5cm 0.5cm 1cm 1.5cm]{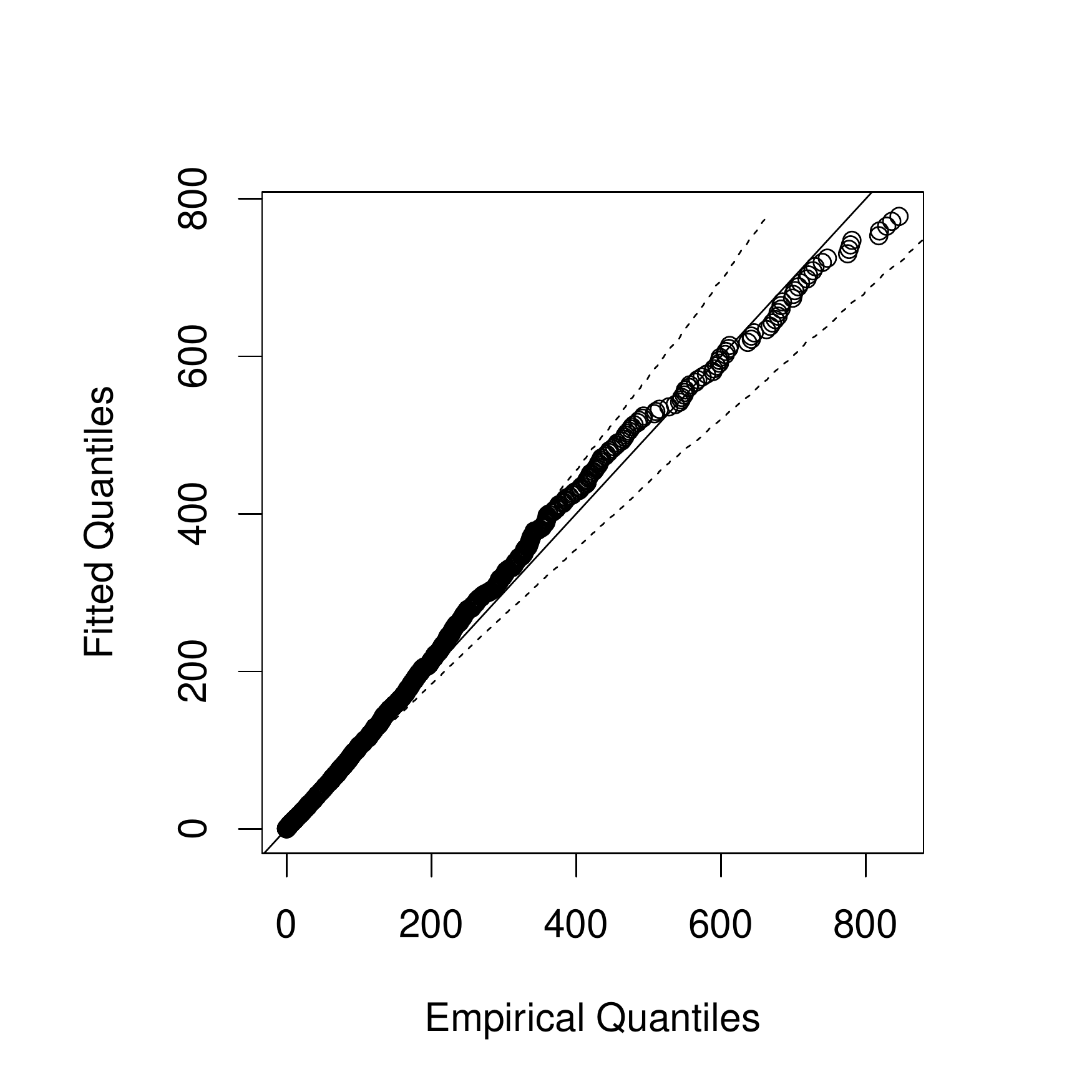}
    \end{subfigure}
       \caption{QQ-plots for a D-GPD fitted to the frequencies of the $5588$ most frequent words in a British corpus. On the right, only percentiles below $99\%$ are plotted. Dashed lines denote pointwise $90\%$ confidence intervals.
       % for empirical quantile estimates.
  }
  \label{fig:qqBritish}
        \end{figure}

    \begin{figure}
\center
\begin{subfigure}[b]{0.45\textwidth}
        \includegraphics[width=\textwidth, trim= 1.5cm 0.5cm 1cm 1.5cm]{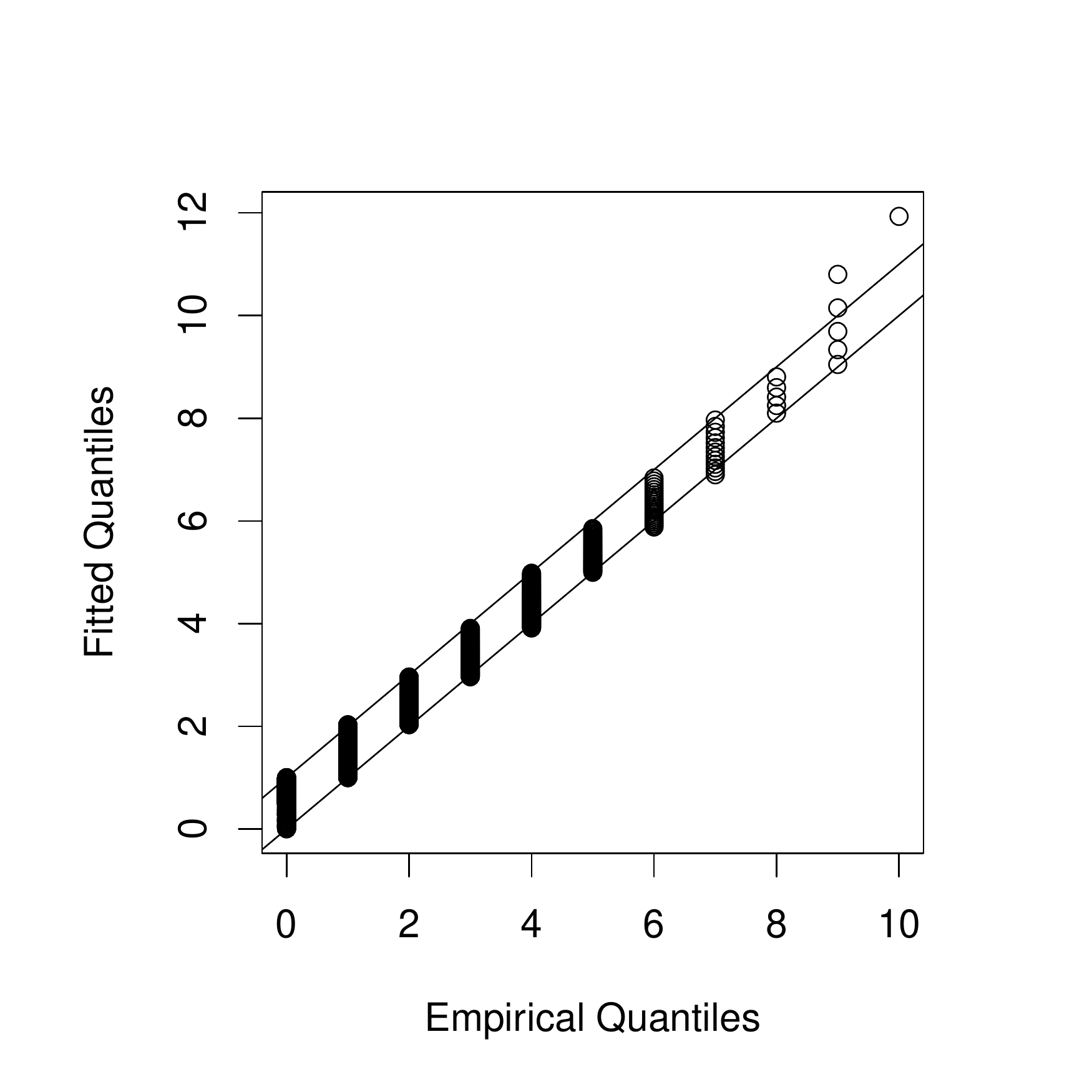}
    \end{subfigure}
\begin{subfigure}[b]{0.45\textwidth}
        \includegraphics[width=\textwidth, trim= 1.5cm 0.5cm 1cm 1.5cm]{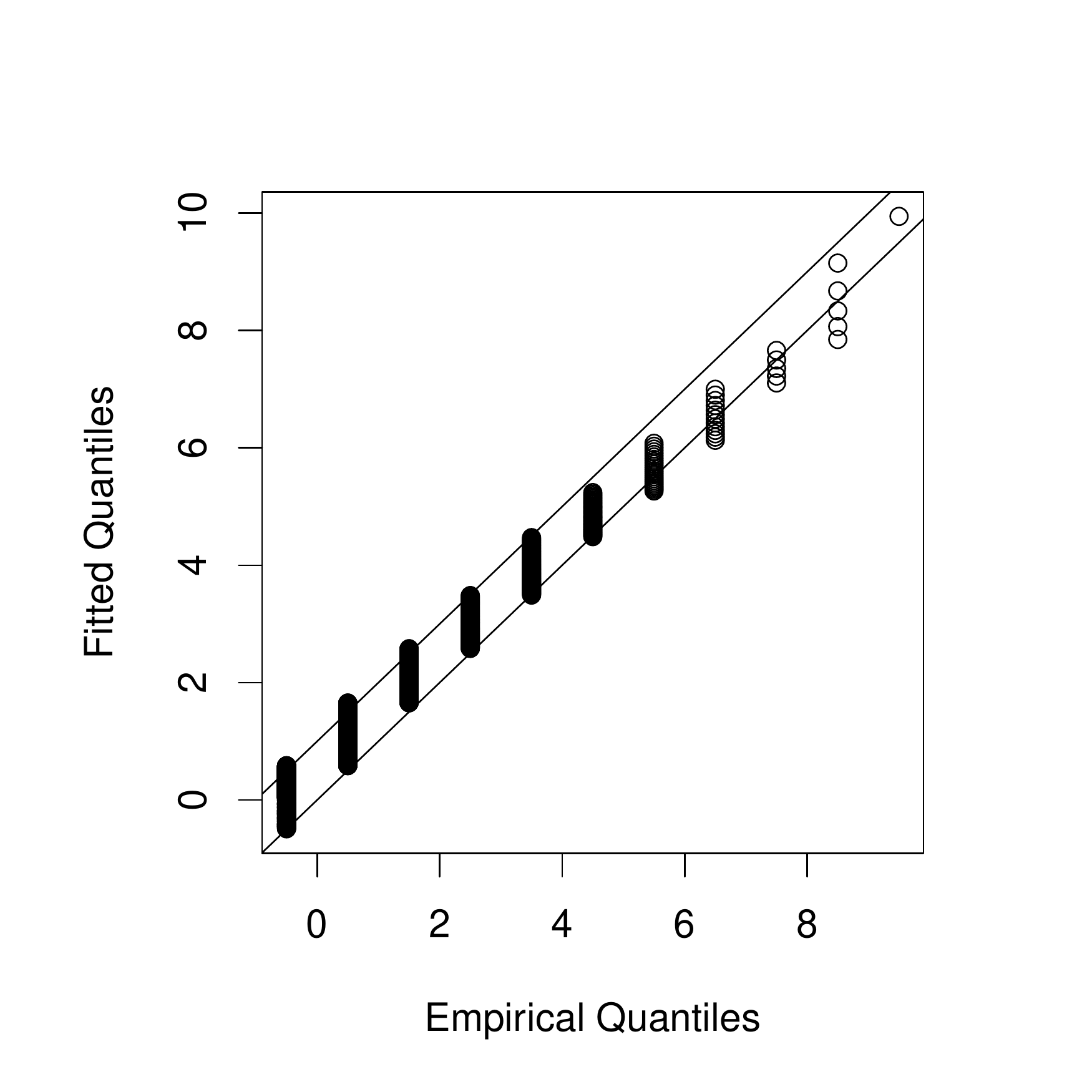}
    \end{subfigure}                
     \caption{QQ-plots for a D-GPD (left) and GPD with $\delta=\frac{1}{2}$ (right) fitted to the length of the $2875$ French words consisting of $15$ letters or more. A good fit occurs when the lowest part of each accumulation of points is close to the line.}
  \label{fig:qq15French}
\end{figure}

We consider a data set based on the British National corpus counting written and spoken English words (\cite{British2007}, [2007]) and are interested in modeling the frequency of the most popular words. The six most frequent words in the corpus are: ``the", ``of", ``and", ``a", ``in" and ``to." Let $X$ be the frequency at which a word occurs. We select a large threshold $u$ and fit a GPD, a D-GPD and a GZD to $X-u\mid X\geq u$ by maximum likelihood. The GPD is implemented with continuity correction ($\delta=\frac{1}{2}$) and without ($\delta=0$). Selecting an appropriate threshold is crucial when estimating high quantiles and can be based on techniques such as mean residual plots (see e.g.\ \cite{Davison1990}). As our focus here is rather on describing the tail distribution of these word frequencies, we choose a relatively low threshold exceeded by $5588$ observations. A frequency plot of $X\mid X\geq u$ is displayed on the right-hand side in Figure~\ref{fig:tablePlotSim}. 

$ $

\begin{table}[h]
\small
\begin{center}
\begin{tabular}{l|l*{5}{c}r}
             Model & p-value &  NLL & $ \xi$ & $ \sigma $    \\ 
\hline
\hline
Word Frequency & & & & &\\
\hline
 British & & & & \\ 
  D-GPD & $ \textbf{0.56} $ & $ 27896.6 $ & $0.88_{[0.84,0.93]}$ & $22.38_{[21.42,23.34]}$ \\ 
 GZD & $ \textbf{0.56} $ & $ 27896.6 $ & $0.88_{[0.84,0.93]}$ & $22.82_{[21.87,23.77]}$ \\ 
 $ \text{GPD}_{\delta=\frac{1}{2}} $ & $ \textbf{0.55} $ & $  $ & $0.88_{[0.84,0.93]}$ & $22.39_{[21.43,23.34]}$ \\ 
 $\text{GPD}_{\delta=0}$ & $ 0.00 $ & $  $ & $0.93_{[0.89,0.97]}$ & $20.89_{[19.96,21.81]}$ \\ 
  Negative binomial & $ 0.00 $ & $ 29663.2 $  \\ 
\hline
\hline
Word length & & & & &\\
\hline
  French & & & & \\ 
 D-GPD & $\textbf{0.85}$ & $ 3894.0 $ & $0.02_{[-0.01,0.06]}$ & $1.36_{[1.30,1.43]}$ \\ 
 GZD & $\textbf{0.84} $ & $ 3894.0 $ & $0.02_{[-0.01,0.06]}$ & $1.37_{[1.32,1.43]}$ \\ 
 $ \text{GPD}_{\delta=\frac{1}{2}} $ & $ 0.00 $ & $  $ & $-0.04_{[-0.06,-0.01]}$ & $1.51_{[1.45,1.57]}$ \\ 
  Negative binomial & $0.75$ & \textbf{$3893.9$} & & &\\
\end{tabular}
\caption{
Fit of a GPD, D-GPD and GZD to frequencies of the most frequent words in a British corpus and length of the longest French words. The table displays the p-value of discrete Kolmogorov--Smirnov tests, negative log-likelihood (NLL) and maximum likelihood estimators with $90\%$ confidence intervals. The sample sizes are $5588$ for word frequencies and $2875$ for word lengths. 
}
\label{tab:wfMain}
\end{center}
\end{table}

The D-GPD delivers a good fit as revealed by the QQ-plot in Figure~\ref{fig:qqBritish}; indeed, most observations lie within the pointwise $90\%$ confidence intervals obtained by simulating $2000$ times from the fitted model and computing empirical quantiles at each simulation. Table~\ref{tab:wfMain} presents maximum likelihood estimators for $\sigma$ and $\xi$ with $90\%$ confidence intervals. Apart from the case of the GPD with $\delta=0,$ all models provide similar results as expected from Proposition~\ref{pro:convRatio} when $\sigma$ is large. A Kolmogorov--Smirnov test for discrete data (\cite{Arnold2011}) leads to the same conclusion (in order to perform the test for the GPD, which is a continuous distribution, we assumed that data were rounded realizations of the fitted model). The analysis of frequencies of French words in a collection of books and movie subtitles leads to analogous results (\cite{thesisHitz2017}, Chapter 2).  Since the GZD coincides with a Zipf--Mandelbrot distribution when $\xi>0,$ the above results are consistent with a common hypothesis in linguistic that word frequencies follow a Zipf-type law (see e.g.\ \cite{Booth1967}). 

% TODO: mention neg bin better for word lengths

We now consider the set of $150\,000$ words in the French lexical (\cite{Francais2004}) and denote by $X$ the length of a word. The longest French word, for instance, is ``anticonstitutionnellement", consisting of $25$ letters. Contrary to the word frequencies, this data set contains many tied observations and we want to see if this translates into a marked difference between the methods. We thus fit the models to $X-u\mid X\geq u$ with $u=15$ (the $98\%$ empirical percentile of the data), leaving $2875$ exceedances. The D-GPD and GZD deliver a good fit and similar estimations between each other, and this time they clearly outperform the GPD approximation with $\delta=\frac{1}{2}$ as shown by QQ-plots in Figure~\ref{fig:qq15French} and discrete Kolmogorov--Smirnov tests in Table~\ref{tab:wfMain} (p-values are here computed by Monte Carlo simulation). Notice that the negative binomial also fits well the observations in this case. 

To summarize, we have illustrated the adequacy of the D-GPD and GZD in describing the frequencies of the most common and longest words from large corpora. This data analysis supports the conclusion drawn earlier from the simulated case: the D-GPD and GZD are preferred over the GPD to model extremes of discrete data when tied observations are frequent.

\subsection{Tornadoes}

Accurately assessing the risk of environmental hazards is crucial for insurance companies in particular, and we illustrate here how the D-GPD and GZD approximations can be useful techniques for this purpose.

We consider the data set studied in \cite{Tippett2016} which reports the number of extreme tornadoes per outbreak in the United States between 1965 and 2015. They defined an ``extreme outbreak'' as a sequence of twelve or more tornadoes occurring close to one another in time and that are rated F1 and greater on the Fujita scale, or EF1 and greater on the enhanced Fujita scale \citep{fuhrmann2014ranking}. 

Let $X$ be the number of such tornadoes per extreme outbreak. The authors found that observations from $X-u\mid X\geq u$ for $u=12$  were well modeled by a GPD with linear temporal trend in the scale parameter and continuity correction $\delta=\frac{1}{2}.$ The GPD, however, is not a discrete distribution and the D-GPD and GZD seem appropriate choices for this type of data. We fit these three distributions and find that they all achieve comparable quality of fit as shown in Table~\ref{tab:torn}; linear trend in $\sigma$ leads to significative improvements in each case. 
 
% TODO: cite \cite{Trenary2016} 

\begin{table}[h]
\small
\begin{center}
\begin{tabular}{l|llll}
Model &  NLL & $ \xi$ & $ \sigma_0 $ & $ \sigma_t $    \\ 
\hline
% D-GPD & $ 2910.1 $ & $0.30_{[0.19,0.41]}$ & $7.60_{[6.58,8.62]}$ &  \\ 
%  GZD  & $ 2910.1 $ & $0.30_{[0.19,0.41]}$ & $7.75_{[6.76,8.73]}$ &  \\ 
% $ \text{GPD}_{\delta=\frac{1}{2}} $ & $ 2910.1 $ & $0.30_{[0.19,0.41]}$ & $7.65_{[6.63,8.67]}$ &  \\ 
 $\text{D-GPD}$ & $ 1439.92 $ & $0.27_{[0.16,0.37]}$ & $4.81_{[3.64,5.99]}$ & $6.11_{[3.74,8.48]}$ \\  
 $ \text{GPD}_{\delta=\frac{1}{2}} $  & $ 1439.93 $ & $0.26_{[0.16,0.37]}$ & $4.86_{[3.68,6.04]}$ & $6.13_{[3.75,8.50]}$ \\ 
\end{tabular}

\caption{Fit of a D-GPD and GPD (with linear trend in the scale parameter) to the number of tornadoes per extreme tornado outbreak in the United States.  Only the $435$ outbreaks with more than $12$ tornadoes were considered. The table displays negative log-likelihood (of the discretized model in the case of the GPD) and maximum likelihood estimates with $90\%$ confidence intervals. 
}
\label{tab:torn}
\end{center}
\end{table}

\subsection{Multiple Births}
 We now turn our attention to a data set consisting of very small integer values to see if the D-GPD and GZD can describe tail distributions in this special case. We examine data counting multiple births in the United States from $1995$ to $2014$ (\cite{USMultipleBirth2015}); its frequency table reads

$ $
\begin{center}
\small 
\begin{tabular}{lllll}
            single &  twin & triplet & quadruplet  & quint.\ or more  \\
       \hline
       78\,178\,588     &     2\,500\,340        &   117\,603    &         8\,108       &      1\,353. \\
\end{tabular}
\end{center}
$ $

Let $X$ be the number of children at birth. We fit a right-censored D-GPD, GZD, negative binomial and Poisson distribution to $X^C-u\mid X^C\geq u$ for $u=1$, where $X^C=\min(X,5).$  As shown in the table below, which displays Bayesian information criteria (BIC) and maximum likelihood estimates, the D-GPD and GZD outperform the Poisson and negative binomial distributions, and are useful methods if one must estimate, for example, the probability that an American women delivers sextuplets. 

%and p-values of $\chi^2$ goodness-of-fit tests.

$ $
\begin{center}
\small
\begin{tabular}{l|llll}
          &    BIC & $\xi$ & $\sigma$  \\
          \hline
D-GPD   & $ \textbf{546\,441.2} $ & $0.06_{[0.06,0.06]}$ & $0.30_{[0.30,0.30]}$  \\ 
GZD & $ \textbf{546\,440.6} $ & $0.06_{[0.06,0.07]}$ & $0.32_{[0.32,0.32]}$  \\ 
Negative binomial  & $ 546\,547 $ &&  \\ 
Poisson  & $ 552\,284 $ & &    \\ 
\end{tabular}
\end{center}
$ $

The applicability of the D-GPD and GZD approximations for estimating tail distributions when the selected threshold $u$ is a small integer should be more rigorously explored. Future work could also assess the validity of the approximations in the case $\xi<0$, and compare them to a broader class of discrete distributions such as discrete compound Poisson distributions. Since the D-GPD and GZD delivered similar performances in the data analysis carried out in this article, it would be interesting to further understand how they relate to one another.

%Many real data sets are consisting of few integer values and the D-GPD and GZD offer potential distributions. Future work could explore how 
% methods for estimating the probability of rare events
%establish the assess the ability of these approximations to estimate tail distributions. 

$ $

\paragraph*{Acknowledgements} The first author is grateful to the Berrow Foundation for financial support and would like to thank Robin Evans, Gesine Reinert, David Steinsaltz and Jonathan Tawn for valuable advice and encouragement. This research was partially supported by the ARO grant  W911NF-12-10385.

\small

\bibliographystyle{abbrvnat}
\bibliography{references}

\end{document}